\documentclass{article}

\usepackage{graphicx}
\usepackage{amsmath,amssymb,amsfonts}
\usepackage{amsthm}
\usepackage{authblk}
\usepackage{xcolor}
\usepackage{tikz}

\theoremstyle{plain}
\newtheorem{theorem}{Theorem}
\newtheorem{lemma}[theorem]{Lemma}

\theoremstyle{definition}
\newtheorem{definition}[theorem]{Definition}
\newtheorem{remark}[theorem]{Remark}

\title{Undirected edge geography games on stacked prism graphs}

\author[1]{Tharit Sereekiatdilok\thanks{Email: tharit.yoii@gmail.com}}
\author[1,2]{Panupong Vichitkunakorn\thanks{Email: panupong.v@psu.ac.th}}
\affil[1]{Division of Computational Science, Faculty of Science, Prince of Songkla University, Hat Yai, Songkhla 90110, Thailand}
\affil[2]{Research Center in Mathematics and Statistics with Applications, Prince of Songkla University, Hat Yai, Songkhla 90110, Thailand}
\date{}

\begin{document}

\maketitle

\begin{abstract}
The undirected edge geography is a two-player combinatorial game on an undirected graph. The players start at the root vertex and alternately move the root along an incident edge to its other endpoint and then delete that edge. The first player who has no remaining move is the loser. For positive integers $m$ and $n$ where $m\geq 3$, the stacked prism graph is $SP(m,n)=C_m\square P_n$. In this paper, we completely determine the winner of the game on $SP(2m,n)$ for $m\geq 2$ and $SP(m,2)$ for $m\geq 3$, and provide a winning strategy for the winner.
\end{abstract}



\section{Introduction}
\label{intro} 

The geography game was first introduced in \cite{FRAENKEL1993197} as a two-player game on a directed graph. 
It starts at a vertex $v$ and players alternately move to another vertex along an arrow (a directed edge) and then delete the previously visited vertex.
The first player who has no legal move is the loser.
Later, the game is extended to an undirected graph and the previously visited edge is removed instead of a vertex.
This version of the game is called the \emph{Undirected Edge Geography (UEG)} \cite{Fraenkel1993}.

In \cite{sereekiatdilok2026undirected}, the undirected edge geography game on a grid is considered. The winner at each root vertex is completely determined and a strategy for the winner is provided.
In this work, we consider the game on the \emph{stacked prism} graph $SP(m,n) := C_m \square P_n$ where $m\geq 3$.
The graph is considered as a grid graph with $m+1$ columns and $n$ rows, where the first and the last column are identified.
In particular, we embed $SP(m,n)$ inside a rectangle $[0,m]\times[0,n+1]$ such that $V(SP(m,n)) = \{0,\dots,m\}\times \{1,\dots,n\}$, the vertices $(0,j)$ and $(m,j)$ are identified for all $j$, and two vertices $(i,j)$ and $(i',j')$ are adjacent if and only if $|i-i'|+|j-j'|=1$, where the first coordinates are taken modulo $m$. An example of $SP(6,4)$ is shown in Figure~\ref{fig:SP(6,4)}.

\begin{figure}[!ht]
    \centering
    \includegraphics[width=0.6\textwidth]{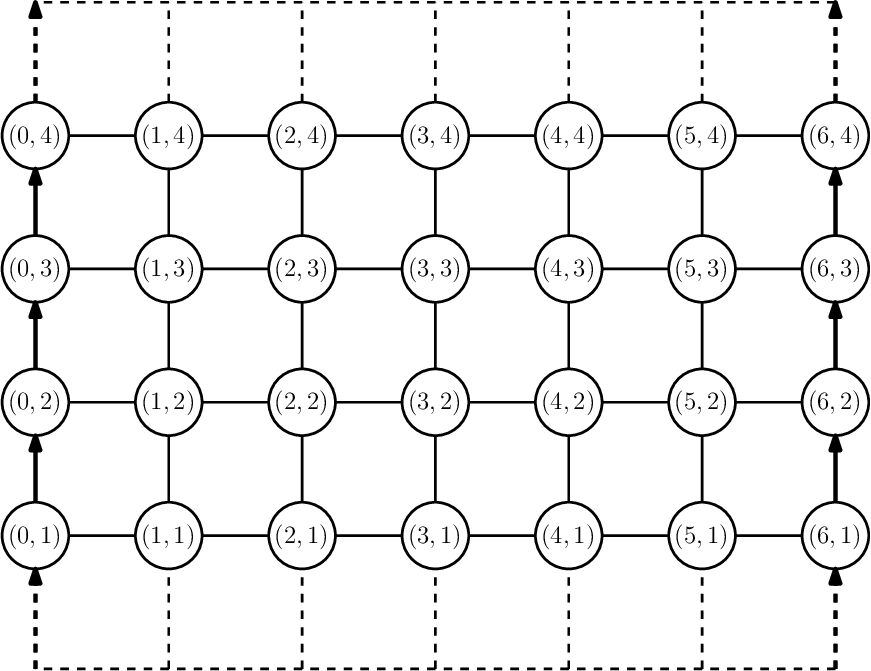}
    \caption{$SP(6,4)$ inside the rectangle $[0,6]\times[0,5]$.}
    \label{fig:SP(6,4)}
\end{figure}

For a graph $G$ with a starting vertex $v$, denoted by a game $(G,v)$, we try to find its \emph{position}, which determines which player has a winning strategy.
A game $(G,v)$ is an \emph{N-position} if the first player has a winning strategy. Otherwise, it is a \emph{P-position}.
The main tool that we use to determine the position of a game in this work is an \emph{even kernel}.

\begin{definition}[\cite{Fraenkel1993}]
    An \emph{even kernel} for a simple graph $G$ is a nonempty set $S \subseteq V(G)$ such that
    \begin{enumerate}
        \item $S$ is independent and
        \item for any vertex $u\notin S$, the number of its neighbors in $S$ is even, i.e., $|N(u)\cap S|$ is even.
    \end{enumerate}
\end{definition}

\begin{theorem}[\cite{Fraenkel1993}]\label{EvenKernel}
    If $S$ is an even kernel for a simple graph $G$ and $v\in S$, then $(G,v)$ is a P-position.
    Moreover, if $G$ is a bipartite graph, then $(G,v)$ is a P-position if and only if $v$ is in an even kernel for a graph $G$.
\end{theorem}

For the rooted undirected edge geography game on grid graphs, we have the following result.

\begin{theorem}[\cite{sereekiatdilok2026undirected}]\label{thm: grid}
    Let $\delta = \gcd(m+1, n+1)$ and $v=(a,b)$. Then $(G_{m\times n}, v)$ is a P-position if and only if $\delta \nmid a$ and $\delta \nmid b$.
\end{theorem}

In this paper, we consider the rooted undirected edge geography game on $SP(2m,n)$ and $SP(m,2)$. 
We get the following main results.

\begin{theorem}\label{thm:main1}
    For positive integers $m,n$ where $m\geq 2$, let $G=SP(2m,n)$, $v=(a,b)\in V(G)$ and $d=\gcd(2m,n+1)$.
    Then $(G,v)$ is a P-position if and only if $d\nmid b$.
\end{theorem}

\begin{theorem}\label{thm:main2}
    Let $m\geq 3$. The game on $SP(m,2)$ is a P-position if and only if $m\equiv 0 \pmod{3}$.
\end{theorem}

\section{Main results}
\label{sec:1}

\subsection{Games on stacked prism graphs}

We first give the condition for the P-positions.

\begin{theorem}\label{thm:main1-1}
    For $m\geq 2$ and $n\geq 1$, let $G=SP(2m,n)$, $v=(a,b)\in V(G)$, $d'=\gcd(m,n+1)$, and $d=\gcd(2m,n+1)$.
    Then $(G,v)$ is a P-position if $d\nmid b$ and $d=d'$.
\end{theorem}
\begin{proof}
    Let $v=(a,b).$
    Assume that $d\nmid b$ and $d=d'$.
    Let 
    $$S := \{(i,j) \in V(G) : d'\nmid j \text{ and } i\equiv\pm j \pmod{2d'}\}.$$
    We will show that $S$ is an even kernel for $G$.

    Extend the cylindrical grid to all integer rows, keeping the first coordinate modulo \(2m\). Define the extension of $S$ using symmetric difference as
    \[
    \qquad \widetilde S=A_+\triangle A_- \quad\text{where}\quad A_\pm=\{(i,j):i\equiv\pm j\pmod{2d'}\}.
    \]
    These congruences are well defined because \(d'\mid m\). Since \(j\equiv-j\pmod{2d'}\) exactly when \(d'\mid j\), the restriction of \(\widetilde S\) to \(V(G)\) is precisely \(S\).
    For any vertex of the extended grid, its neighbors in \(A_+\) occur in pairs: east with south, and west with north. Likewise, its neighbors in \(A_-\) pair east with north, and west with south. Consequently,
    \[
    |N(u)\cap\widetilde S|
    \equiv |N(u)\cap A_+|+|N(u)\cap A_-|
    \equiv0\pmod2.
    \]Moreover, \(\widetilde S\) contains no vertices in rows \(0\) and \(n+1\), since \(d'\mid n+1\). Restricting to \(G\) therefore preserves these neighbor counts.
    Every vertex of \(S\) has \(i+j\) even, so \(S\) is independent. Finally, the hypotheses imply \(d'\nmid b\), hence \(d'>1\) and \((1,1)\in S\). Thus \(S\) is a nonempty independent set with even neighbor counts, and hence an even kernel.
    Figure~\ref{fig:SP(12,8)} shows an example of $S$ for $SP(12,8)$.

    \begin{figure}[!ht]
    \centering
        \includegraphics[width=0.6\textwidth]{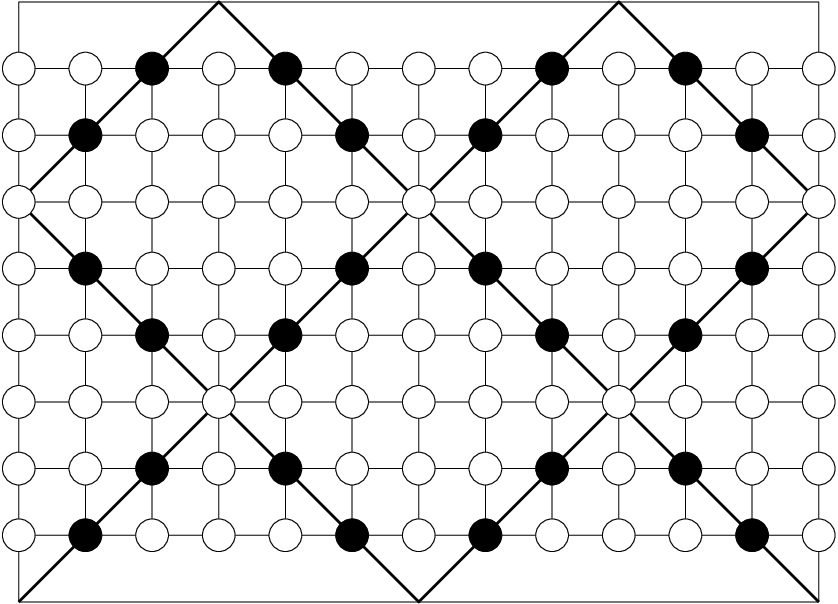}
        \caption{An even kernel for a graph $SP(12,8)$.}
    \label{fig:SP(12,8)}
    \end{figure}

    If $v\in S$, then $(G,v)$ is a P-position by Theorem~\ref{EvenKernel}.
    Otherwise, if $v\notin S$, there is a graph automorphism $f$ such that $f(a,b)=(a',b)\in S$.
    Then $f^{-1}(S)$ is an even kernel of $G$ containing $(a,b)$.
    By Theorem~\ref{EvenKernel}, $(G,v)$ is a P-position.
\end{proof}

\begin{remark}\label{jumping-trail}
    We can consider a set of vertices $S$ in Theorem~\ref{thm:main1-1} as a set of vertices that was visited once by a single ray going out from the bottom-left corner with $45^\circ$ angle to the side of the rectangle. 
    It reflects at the top and bottom sides of the rectangle.  
    But, it jumps from the right side ($i=m$) to the left side ($i=0$) keeping the same slope.
\end{remark}

When $d\neq d'$, the set $S$ in the proof of Theorem~\ref{thm:main1-1} may not be an even kernel.
We then have another construction in the following theorem.

\begin{theorem}\label{thm:main1-2}
    For $m\geq 2$ and $n\geq 1$, let $G=SP(2m,n)$, $v=(a,b)\in V(G)$, $d'=\gcd(m,n+1)$ and $d=\gcd(2m,n+1)$.
    Then $(G,v)$ is a P-position if $d\nmid b$ and $d\neq d'$.
\end{theorem}
\begin{proof}
    Let $v=(a,b).$
    Suppose that $d\nmid b$ and $d\neq d'$.
    Then $d=2d'$.
    Now, we partition the rectangle $[0,2m]\times[0,n+1]$ into small $d\times d$ squares. In each square, it is possible to draw a diamond as described as dashed lines in Figure~\ref{fig:EvenKernelSP}.
    Let $R$ be the set of vertices of $G$ in the interior of any diamond.
    Let $D$ be the set of vertices $(i,j)\in R$ such that $i+j$ and $d'$ have different parities.
   
    Next, we will show that $D$ is an even kernel. It is easy to see that $D$ is an independent set.
    To show that any vertices not in $D$ have an even number of neighbors in $D$, we consider $u\notin D$.
    \begin{itemize}
        \item If $u$ is in $R$, then it has $4$ neighbors in $D$.
        \item If $u$ is on the boundary of a diamond, then it has $2$ neighbors in $D$.
        \item Otherwise, it has no neighbor in $D$.
    \end{itemize} 
    So, $D$ is an even kernel.
    Note that for all $j$ where $d\nmid j$, there exists a vertex $(i,j)$ in $D$. 

    If $v\in D$, we are done.
    If $v\notin D$, there is an automorphism $f$ on $G$ that $f(a,b)=(a',b)\in D$.  
    Similar to the proof of Theorem~\ref{thm:main1-1}, $(G,v)$ is a P-position.
    \begin{figure}[!ht]
    \centering
        \includegraphics[width=0.5\textwidth]{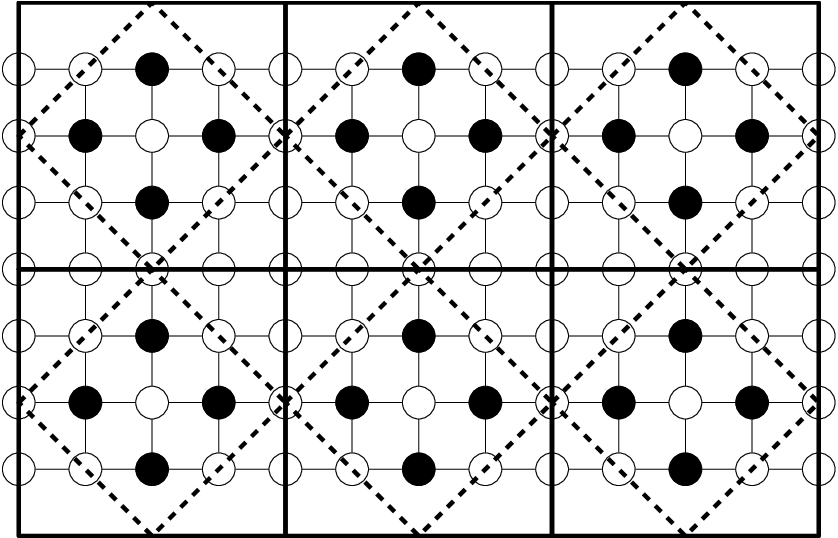}
        \caption{An even kernel for a graph $SP(12,7)$.}
    \label{fig:EvenKernelSP}
    \end{figure}
\end{proof}

From Theorems~\ref{thm:main1-1} and \ref{thm:main1-2}, we can conclude that
$(G,v)$ is a P-position if $d\nmid b$.
Next, we will show that $(G,v)$ is an N-position if $d\mid b$.

From now on, we consider only the right half $[m,2m]\times[0,n+1]$ of the whole rectangle $[0,2m]\times[0,n+1]$. 
In this rectangle, we define a \emph{diagonal trail} as a single ray going out from a corner with $45^\circ$ angle to the side of the rectangle. It reflects at all sides of the rectangle. 

From the proof of Lemma~4 in \cite{sereekiatdilok2026undirected}, we have the following lemma.

\begin{lemma}\label{lem:fulltrail}
    Let $T := \{(i,j) \in\{m,\dots,2m\}\times\{1,\dots,n\} : i\pm j\equiv m \pmod{2d'}\}$ and $d'=\gcd(m,n+1)$. 
    In the rectangle $[m,2m]\times[0,n+1]$ where rays reflect off the boundary of the rectangle, the diagonal trail starting at $(m,0)$ making $45^\circ$ to the bottom side visits every vertex in $T$ and ends at another corner of the rectangle. 
\end{lemma}
\begin{proof}
    Consider the rectangle $[m,2m] \times [0,n+1]$ where a trail reflects off the boundary.
    From the proof of Lemma~4 in \cite{sereekiatdilok2026undirected}, the $45^\circ$  diagonal trail at $(m,0)$ visits $(i,j)$ if there is $a,b\in \mathbb{Z}$ such that $(2am\pm i)-m = 2b(n+1) \pm j$, where the two $\pm$ are independent. This reduces to $i \pm j = m + 2b(n+1)-2am$. 
    
    Let $(i,j)\in T$. We have $i\pm j \equiv m \pmod{2d'}$. Since $d'=\gcd(m,n+1)$, there is $a',b'\in\mathbb{Z}$ such that $d' = a'm + b'(n+1)$. This gives $i\pm j = m + 2ka'm + 2kb'(n+1)$ for some $k\in\mathbb{Z}$. So we can pick $a=-ka'$ and $b=kb'$. Hence, $i \pm j = m + 2b(n+1)-2am$ as desired.

    Assume that the trail visits only the corner $(m,0)$, i.e., none of $(m,n+1)$, $(2m,0)$ or $(2m,n+1)$ are in $T$.
    Then all of the following hold.
    \begin{align}
        m+n+1 &\not\equiv m \pmod{2d'}. \label{Eq1}\\
        2m &\not\equiv m \pmod{2d'}. \label{Eq2}\\
        2m+n+1 &\not\equiv m \pmod{2d'}. \label{Eq3}
    \end{align} 
    Since $d'\mid m$ and by Equation~\eqref{Eq2}, we have $m \equiv d' \pmod{2d'}$.
    By Equation~\eqref{Eq3}, $n+1 \not\equiv d' \pmod{ 2d'}$. Since $d'\mid n+1$, we also have $n+1 \equiv 0 \pmod{2d'}$.
    Thus, $m+n+1 \equiv d' \equiv m \pmod{2d'}$, a contradiction to Equation~\eqref{Eq1}.
    So the trail visits a corner $(m,n+1)$ or $(2m,0)$ or $(2m,n+1)$.
\end{proof}

\begin{theorem}\label{thm:trail}
    For $m\geq 2$ and $n\geq 1$, let $d = \gcd(2m,n+1)$ and $d\mid b$.
    There is a diagonal trail $\tau'$ on a rectangle $[m,2m] \times [0,n+1]$ that starts at a corner vertex $(m,0)$ or $(m,n+1)$, reflects at any side of the rectangle, and visits a vertex $(m,b)$.    
\end{theorem}
\begin{proof}
    Let $d'=\gcd(m,n+1)$.
    Then $d'\mid 0$ and $d'\mid b$.
    Let $T$ be the set of vertices obtained from Lemma~\ref{lem:fulltrail}.
    
    If $d\neq d'$, then $d =2d'$. Thus $(m,b)\in T$ and we get $\tau'$ as a subtrail of the diagonal trail starting at $(m,0)$ and visiting $(m,b)$.
    
    If $d = d'$, then we have two cases: $b\equiv 0 \pmod{2d'}$ and $b\equiv d' \pmod{2d'}$.
    For $b\equiv 0 \pmod{2d'}$, we have $(m,b)\in T$ and we are done.
    On the other hand, if $b\equiv d' \pmod{2d'}$, then we let $n' = \frac{n+1}{d'}$.
    We will show that $n'$ is odd.
    Since $d=d'$, we have $\frac{2m}{d}=\frac{2m}{d'}= 2\cdot\frac{m}{d'}$ is even.
    If $n'=\frac{n+1}{d'}=\frac{n+1}{d}$ is even, then $d = \gcd(2m,n+1)\geq 2d$, which is a contradiction. 
    So $n'$ is odd, that is, $n+1\equiv d' \pmod{2d'}$.
    We consider a trail $\tau'$ that starts at $(m,n+1)$. 
    From Lemma~\ref{lem:fulltrail} with vertical reflection, the trail $\tau'$ visits every vertex in $\{(i,j) \in\{m,\dots,2m\}\times\{1,\dots,n\} : i\pm j\equiv m+d' \pmod{2d'}\}$.
    Since $b\equiv d' \pmod{2d'}$, a trail $\tau'$ visits a vertex $(m,b)$.
    
    Hence there is a trail $\tau'$ starting at $(m,0)$ or $(m,n+1)$ and visiting $(m,b)$.
\end{proof}

The trail $\tau'$ in Theorem~\ref{thm:trail} divides the rectangles into disjoint regions. We would like to label them with \emph{Positive} label and \emph{Negative} label so that any adjacent regions must have distinct labels. We call it a \emph{proper} labeling.

\begin{theorem}
Consider the regions obtained from a trail $\tau'$ in Theorem~\ref{thm:trail} starting from $(m,0)$ or $(m,n+1)$ and ending at $(m,b)$. There is a proper labeling $L'$ such that all regions adjacent to the top, bottom, or right boundary have the Negative label. 
\end{theorem}
\begin{proof}
    Without loss of generality, we assume that $\tau'$ starts from $(m,0)$.
    If $(m,b)$ is the first vertex on the left boundary that $\tau'$ visits, then we label the regions bounded by $\tau'$ with Positive label, as shown in Figure~\ref{fig:Labeling}.
    \begin{figure}[!ht]
    \centering
        \includegraphics[width=0.4\textwidth]{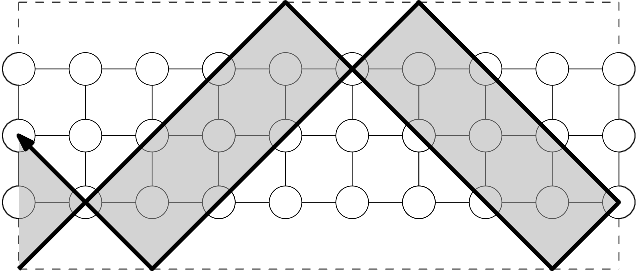}
    \caption{A labeling of regions of the rectangle $[9,18]\times[0,4]$.}
    \label{fig:Labeling}
    \end{figure}
    
    If $(m,b)$ is not the first vertex on the left boundary that $\tau'$ visits, then we will prove the statement by induction.
    Consider the last vertex $(m,b')$ on the left boundary where $\tau'$ visits before a vertex $(m,b)$. 
    We call the subtrail from $(m,0)$ to $(m,b')$ as $\tau''$.
    Let $L''$ be the labeling of the regions obtained from $\tau''$.
    Assume that $L''$ is a proper labeling that satisfies the condition of the theorem.
    
    Note that the labeling $L'''$ of the regions obtained from the trail $\tau'-\tau''$ also satisfies the theorem, by using the idea of the base case.
    We label the regions obtained from $\tau'$ by using $L''$ and $L'''$ as follows.
    \begin{enumerate}
        \item[$\bullet$] If a region get the labels given by $L''$ and $L'''$ are same, we label this region with Negative label.
        \item[$\bullet$] Otherwise, we label the region with Positive label.
    \end{enumerate}
    See Figure~\ref{fig:Labeling2} for an example.

    \begin{figure}[!ht]
    \centering
        \includegraphics[width=0.6\textwidth]{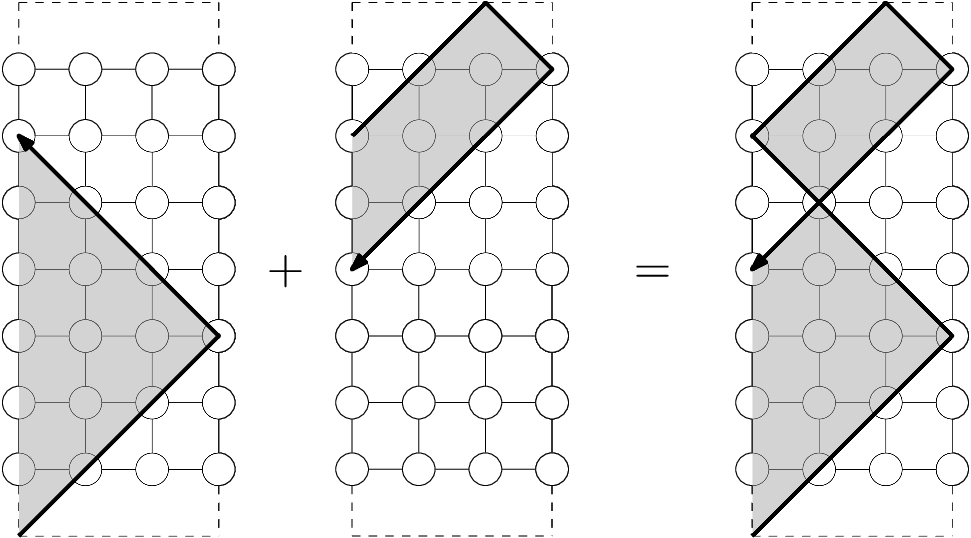}
    \caption{A labeling $L'$ obtained from $L''$ and $L'''$.}
    \label{fig:Labeling2}
    \end{figure}
    
    Since the labeling $L''$ and $L'''$ assign Negative label to all regions adjacent to top, bottom or right boundary, the label $L'$ also satisfies this condition.

    To prove that any two adjacent regions have distinct labels by $L'$, we consider a segment shared by the two regions.
    If it is in $\tau''$, the two regions separated by the segment get different labels by $L''$. On the other hand, the two regions must get the same labels by $L'''$. By the definition of $L'$, they get different labels.
    Similarly, if the segment is in $\tau'-\tau''$, the two regions separated by the segment get the same label by $L''$ but get different labels by $L'''$. So, they get different labels by $L'$.
    
    Therefore, we have a proper labeling $L'$ such that regions on the boundaries other than the left side are negative. 
\end{proof}

\begin{remark}\label{rem:Labeling}
    We can reflect the trail $\tau'$ and the labels of all the regions of $[m,2m]\times [0,n+1]$, by the line $x=m$, to the rectangle $[0,m]\times [0,n+1]$. As a result, we get a proper labeling $L$ on the regions of the full rectangle $[0,2m]\times [0,n+1]$ such that all regions on the boundaries get Negative label, as shown in Figure~\ref{fig:Labeling3}.
    \begin{figure}[!ht]
    \centering
        \includegraphics[width=0.3\textwidth]{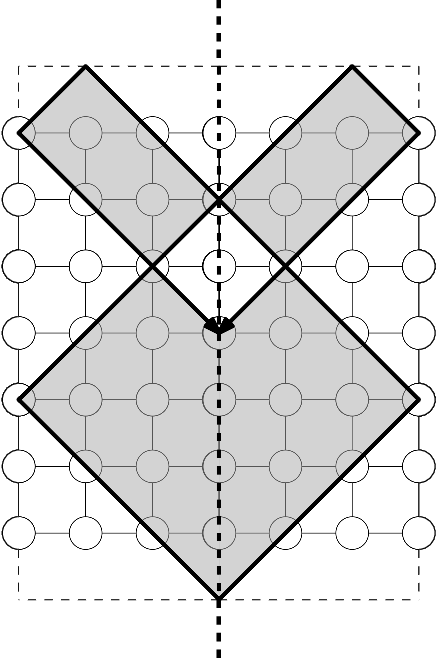}
    \caption{The labeling $L$ obtained from $L'$ by reflection.}
    \label{fig:Labeling3}
    \end{figure}
\end{remark}
    
\begin{theorem}\label{thm:main1-3}
    For $m\geq 2$ and $n\geq 1$, let $G=SP(2m,n)$, $v=(a,b)\in V(G)$, and $d=\gcd(2m,n+1)$.
    Then $(G,v)$ is an N-position if $d\mid b$.
\end{theorem}
\begin{proof}
    Let $d'=\gcd(m,n+1)$ and $d\mid b$.
    To simplify the proof, we consider the game at a vertex $v=(m,b)$.
    Let $u$ be a neighbor of $v$ in a positive region. Define $S$ to be the set of vertices $(i,j)$ in positive regions of the labeling $L$ in Remark~\ref{rem:Labeling} for which $i+j$ and $m+b$ have different parity. Then $u\in S$.
    We show that $S$ is an even kernel of $G-uv$.
    
    Since all vertices in $S$ have the same parity, $S$ is an independent set.
    Let $x\notin S$.
    Then we will show the number of its neighbors in $S$ is even, in different cases as follows.
    \begin{enumerate}
        \item[$\bullet$] If $x$ is in a positive region, then it has $4$ neighbors in $S$.
        \item[$\bullet$] If $x$ is in a negative region, then it has no neighbors in $S$.
        \item[$\bullet$] If $x\neq v$ is on a segment separating regions, then it has $2$ neighbors in $S$.
        \item[$\bullet$] If $x=v$, then it has one or three neighbors in $S$ in $G$, one of which is $u$. Hence it has $0$ or $2$ neighbors in $S$ in $G-uv$.
    \end{enumerate}
    Thus, $S$ is an even kernel of $G-uv$ containing $u$. After Player 1 moves from $v$ to $u$ and deletes $uv$, the resulting position $(G-uv,u)$ is a P-position by Theorem~\ref{EvenKernel}.
    Therefore, Player 1 has a winning strategy on $(G,v)$.
    Hence, $(G,v)$ is an N-position.
\end{proof}

From Theorems~\ref{thm:main1-1}, \ref{thm:main1-2} and \ref{thm:main1-3}, we immediately get Theorem~\ref{thm:main1}.

\subsection{Games on prism graphs}

Since $SP(2m,n)$ is a bipartite graph, by Theorem~\ref{EvenKernel}, we can ensure that there exists an even kernel for determining a position for this graph.
However, for the case of $SP(m,n)$ where $m$ is odd, the graph is not bipartite. So it might not be possible to find an even kernel even though it is a P-position, which makes it not easy to determine the position of the game.

In this section, we consider the game on (stacked) prism graphs $SP(m,2):=C_m\square P_2$ for $m\geq 3$.
To simplify the proof, we draw a graph $SP(m,2)$ and label vertices as shown in Figure~\ref{fig:SP(m,2)}.
Since the graph is vertex-transitive, we can consider only the game that starts at Vertex~$1$.  

\begin{figure}[!ht]
    \centering
    \begin{tikzpicture}[x=0.7cm,y=0.6cm]
        \coordinate (t1) at (0,1.6);
        \coordinate (t2) at (2.5,1.6);
        \coordinate (t3) at (5,1.6);
        \coordinate (tm) at (10,1.6);
        \coordinate (b1) at (0,0);
        \coordinate (b2) at (2.5,0);
        \coordinate (b3) at (5,0);
        \coordinate (bm) at (10,0);
        \draw (t1) to[out=80,in=100,looseness=0.5] (tm);
        \draw (b1) to[out=-80,in=-100,looseness=0.5] (bm);
        \draw (t1) -- (t2) -- (t3) -- (6.2,1.6);
        \draw (7.8,1.6) -- (tm);
        \draw (b1) -- (b2) -- (b3) -- (6.2,0);
        \draw (7.8,0) -- (bm);
        \draw (t1) -- (b1);
        \draw (t2) -- (b2);
        \draw (t3) -- (b3);
        \draw (tm) -- (bm);
        \node[circle,draw,fill=white,minimum size=6mm,inner sep=1pt] at (t1) {$\overline{1}$};
        \node[circle,draw,fill=white,minimum size=6mm,inner sep=1pt] at (t2) {$\overline{2}$};
        \node[circle,draw,fill=white,minimum size=6mm,inner sep=1pt] at (t3) {$\overline{3}$};
        \node[circle,draw,fill=white,minimum size=6mm,inner sep=1pt] at (tm) {$\overline{m}$};
        \node[circle,draw,fill=white,minimum size=6mm,inner sep=1pt] at (b1) {$1$};
        \node[circle,draw,fill=white,minimum size=6mm,inner sep=1pt] at (b2) {$2$};
        \node[circle,draw,fill=white,minimum size=6mm,inner sep=1pt] at (b3) {$3$};
        \node[circle,draw,fill=white,minimum size=6mm,inner sep=1pt] at (bm) {$m$};
        \node at (7,1.6) {$\cdots$};
        \node at (7,0) {$\cdots$};
    \end{tikzpicture}
    \caption{$SP(m,2)$}
    \label{fig:SP(m,2)}
\end{figure}

In the following displayed sets, vertex labels are read modulo $m$. For example, $m+1=1$ and $\overline{m+1}=\overline{1}$. Here, we interpret $\{0,\dots,k-1\}$ as the empty set when $k=0$.

\begin{theorem}\label{thm4}
    Let $m\geq 3$. The game on $SP(m,2)$ is an N-position if $m\equiv 1,2 \pmod{3}$.
\end{theorem}
\begin{proof}
    Let $G=SP(m,2)$. 
    We consider the cases $m=1+6k$, $m=2+6k$, $m=4+6k$, and $m=5+6k$. We aim to find a winning strategy of Player 1.

    \begin{description}
        \item[Case 1: $m=2+6k$.] Player 1 moves from Vertex $1$ to Vertex $\overline{1}$. An even kernel of $G-(1,\overline{1})$ containing Vertex $\overline{1}$ is:
        $$
            \{\overline{1}\}\cup\{\overline{6i+3},6i+4,6i+6,\overline{6i+7} : i\in\{0,\dots,k-1\}\}.
        $$

        \item[Case 2: $m=4+6k$.] Player 1 moves from Vertex $1$ to Vertex $2$. An even kernel of $G-(1,2)$ containing Vertex $2$ is:
        $$
            \{\overline{1},2,4\}\cup\{\overline{6i+5},\overline{6i+7},6i+8,6i+10 : i\in\{0,\dots,k-1\}\}.
        $$

        \item[Case 3: $m=5+6k$.] Player 1 moves from Vertex $1$ to Vertex $2$. An even kernel of $G-(1,2)$ containing Vertex $2$ is:
        $$
            \{1,2,\overline{3},\overline{5}\}\cup\{6i+6,6i+8,\overline{6i+9},\overline{6i+11} : i\in\{0,\dots,k-1\}\}.
        $$

        \item[Case 4: $m=1+6k$.] Player 1 moves from Vertex $1$ to Vertex $2$. We consider two possible moves for Player 2, as follows.
        \begin{description}
            \item[Case 4.1: Player 2 moves to Vertex $\overline{2}$.] Player 1 moves to Vertex $\overline{1}$. An even kernel of $G - \{(1,2),(2,\overline{2}),(\overline{2},\overline{1})\}$ containing Vertex $\overline{1}$ is:
            $$
                \{\overline{1}\}\cup\{\overline{6i+2},\overline{6i+4},6i+5,6i+7 : i\in\{0,\dots,k-1\}\}.
            $$

            \item[Case 4.2: Player 2 moves to Vertex $3$.] Player 1 moves to Vertex $4$. An even kernel of $G - \{(1,2),(2,3),(3,4)\}$ containing Vertex $4$ is:
            $$
                \{1,4,\overline{5},\overline{7}\}\cup\{6i+8,6i+10,\overline{6i+11},\overline{6i+13} : i\in\{0,\dots,k-1\}\}
            $$    
        \end{description}
    \end{description}
    By Theorem~\ref{EvenKernel}, the game on $SP(m,2)$ is an N-position if $m=1+6k$, $m=2+6k$, $m=4+6k$, or $m=5+6k$.
\end{proof}

In Theorem~\ref{thm4}, we are able to find an even kernel after a few steps.
However, we will show that the game on $SP(3,2)$ is a P-position, but there is no even kernel of $SP(3,2)$.

\begin{theorem}\label{thm5}
    The game on $SP(3,2)$ is a P-position.
\end{theorem}
\begin{proof}
    We aim to find a winning strategy of Player 2. Since the game starts at Vertex $1$, we can divide the proof into two cases.

    \begin{description}
        \item[Case 1: Player 1 moves to vertex $2$.] Player 2 moves to Vertex $3$ and gets an even kernel $\{\overline{1},2,3\}$.

        \item[Case 2: Player 1 moves to Vertex $\overline{1}$.] Player 2 moves to Vertex $\overline{2}$.
        \begin{description}
            \item[Case 2.1: Player 1 moves to Vertex $\overline{3}$.] Player 2 responds by moving to Vertex $\overline{1}$. Player 1 cannot make a next move, so Player 2 wins.
            \item[Case 2.2: Player 1 moves to Vertex $2$.] Player 2 responds by moving to Vertex $1$ and gets an even kernel $\{1,2\}$.
        \end{description}
    
    \end{description}
    From all cases, Player 2 has a winning strategy.
    Therefore, the game on $SP(3,2)$ is a P-position.
\end{proof}

\begin{theorem}\label{thm6}
    Let $m\geq 3$. The game on $SP(m,2)$ is a P-position if $m\equiv 0 \pmod{3}$.
\end{theorem}
\begin{proof}
    Let $G=SP(m,2)$.
    We aim to find a winning strategy of Player 2.

    \begin{description}
        \item[Case 1: $m=6k$.] An even kernel of $SP(m,2)$ containing Vertex $1$ is:
        $$
            \{6i+1,\overline{6i+2},\overline{6i+4},6i+5 : i\in\{0,\dots,k-1\}\}.
        $$
        
        \item[Case 2: $m=3+6k$.] The case of $SP(3,2)$ is already proved in Theorem~\ref{thm5}.
        Next, we consider the case $m>3$.
        Without loss of generality, we can only consider two subcases where Player 1 moves to Vertex $2$ or Vertex $\overline{1}$.

        \begin{description}
            \item[Case 2.1: Player 1 moves to Vertex $2$.] Player 2 will move to Vertex $3$ gets an even kernel of $G-\{(1,2),(2,3)\}$ as:
            $$
                \{\overline{1},2,3\} \cup \{\overline{6i+4},\overline{6i+6},6i+7,6i+9 : i\in\{0,\dots,k-1\}\}.
            $$
        
            \item[Case 2.2: Player 1 moves to Vertex $\overline{1}$.] Player 2 then moves to Vertex $\overline{2}$. Now, Player 1 only has two choices.

            \begin{description}
                \item[Case 2.2.1: Player 1 moves to vertex $2$.] Player 2 moves to Vertex $3$ and gets an even kernel of $G-\{(1,\overline{1}),(\overline{1},\overline{2}),(\overline{2},2),(2,3)\}$ as:
                $$
                    \{2,\overline{2},3\} \cup \{6i+5,\overline{6i+6},\overline{6i+8},6i+9 : i\in\{0,\dots,k-1\}\}.
                $$

                \item[Case 2.2.2: Player 1 moves to Vertex $\overline{3}$.] Player 2 moves to Vertex $\overline{4}$ and an even kernel of $G-\{(1,\overline{1}),(\overline{1},\overline{2}),(\overline{2},\overline{3}),(\overline{3},\overline{4})\}$ as:
                $$
                    \{\overline{1}\} \cup \{\overline{6i+4},6i+5,6i+7,\overline{6i+8} : i\in\{0,\dots,k-1\}\}.
                $$
            \end{description}
        \end{description}
    \end{description}
    From all cases, Player 2 always has a winning strategy.
    Hence the game on $SP(m,2)$ is a P-position if $m=6k$ or $m=3+6k$.
\end{proof}

From Theorems~\ref{thm4}, \ref{thm5} and \ref{thm6}, we immediately get Theorem~\ref{thm:main2}.

We see that, even in a simple case of $n=2$, it is not easy to determine the position of the game. A similar analysis gets extensively more complex when $n>2$.

\section*{Acknowledgements}

Tharit Sereekiatdilok was partially supported by the Development and Promotion of Science and Technology Talents Project (DPST). The authors would like to thank the anonymous reviewers for their valuable comments and constructive suggestions, which significantly improved the quality of this manuscript.

\bibliographystyle{amsplain}
\bibliography{ref}

@ARTICLE{FRAENKEL1993197,
title = {Geography},
journal = {Theoretical Computer Science},
volume = {110},
number = {1},
pages = {197-214},
year = {1993},
issn = {0304-3975},
doi = {10.1016/0304-3975(93)90356-X},
author = {Aviezri S Fraenkel and Shai Simonson}
}

@ARTICLE{Fraenkel1993,
title = {Undirected edge geography},
journal = {Theoretical Computer Science},
volume = {112},
number = {2},
pages = {371-381},
year = {1993},
issn = {0304-3975},
doi = {10.1016/0304-3975(93)90026-P},
author = {Aviezri S Fraenkel and Edward R Scheinerman and Daniel Ullman}
}

@article{sereekiatdilok2026undirected,
	title={Undirected edge geography games on grids},
	author={Sereekiatdilok, Tharit and Vichitkunakorn, Panupong},
	journal={Discrete Applied Mathematics},
	volume={393},
	pages={280--286},
	year={2026},
	publisher={Elsevier}
}

\end{document}